\RequirePackage{fix-cm}
\documentclass[smallextended,envcountsect,envcountreset]{svjour3} 
\smartqed  % flush right qed marks, e.g. at end of proof
\usepackage{graphicx}
\usepackage{lineno}
\usepackage{color,url}
\usepackage{enumitem}
\usepackage{amssymb,amsmath,amsfonts,geometry}

\usepackage{cite}
\usepackage{dsfont}

\newcommand{\tto}{\rightrightarrows}

\newcommand{\argmin}{\mathop{\rm argmin}}
 \newcommand{\eto}{\overset{e}{\to}}
\newcommand{\Rex}{\mathbb{R}\cup\{+\infty \}}
\newcommand{\eint}[1]{{#1}_{as}}

\let\epsilon\varepsilon
 
\newcommand{\MoreauYosida}[2]{\operatorname{e}_{#2} #1}
\usepackage{caption}
\begin{document}

\title{Ergodic Approach to Robust Optimization  and Infinite Programming Problems  %\thanks{Grants or other notes
%about the article that should go on the front page should be
%placed here. General acknowledgments should be placed at the end of the article.}
}
%\subtitle{with Applications to scenario approach and }

%\titlerunning{Short form of title}        % if too long for running head

\author{Pedro P\'erez-Aros\thanks{This work was partially supported by CONICYT Chile under grants Fondecyt Regular 1190110 and Fondecyt Regular 1200283.}%etc.
}

%\authorrunning{Short form of author list} % if too long for running head

\institute{Instituto de Ciencias de la Ingenier\'ia,
	Universidad de O'Higgins  \at
              Libertador Bernardo O'Higgins 611, Rancagua, Regi\'on del Libertador Gral. Bernardo O’Higgins, Chile \\
              \email{pedro.perez@uoh.cl}           %  \\
%             \emph{Present address:} of F. Author  %  if needed
}

\date{Received: date / Accepted: date}
% The correct dates will be entered by the editor

\maketitle

\begin{abstract}
In this work, we show the consistency of an approach for solving robust optimization problems using sequences of sub-problems generated by ergodic measure preserving transformations. 

The main result of this paper is that the minimizers and the optimal value of the sub-problems converge, in some sense, to the minimizers and the optimal value of the initial problem, respectively. Our result particularly implies the consistency of the scenario approach for nonconvex optimization problems.  Finally, we show that our method can also be used to solve infinite programming problems.

\keywords{Stochastic optimization \and scenario approach \and robust optimization \and
	epi-convergence \and ergodic theorems.}
% \PACS{PACS code1 \and PACS code2 \and more}
\subclass{MSC 90C15 \and 90C26  \and 90C90 \and 60B11}
\end{abstract}

\section{Introduction}
Robust optimization (RO)  corresponds to a field of mathematical programming dedicated to the study of problems with uncertainty. In this class of models, the constraint set is given by the set of points, which satisfy all (or in the presence of measurability, \emph{almost all}) possible cases.   Roughly speaking, an RO problem corresponds to the following mathematical optimization model 
\begin{equation}\label{firstproblem}
\begin{aligned}
&\min g(x)&\\
\textnormal{s.t. } &x\in M(\xi), &\text{almost surely }\xi \in \Xi,
\end{aligned}
\end{equation}
where  $X$ is a Polish space, $(\Xi, \mathcal{A},\mathbb{P})$ is a probability space,  $M: \Xi \tto X$ is a measurable multifunction with closed values and  $g: X \to \Rex$ is  a lower semicontinuous   function.  We refer to  \cite{Milad_Ackooij_2015,Ackooij_Danti_Fra_2018,MR2834084,MR2546839,MR3242164,MR1375234,MR2232597} and the references therein for more details and applications.

When the number of possible scenarios $\xi \in \Xi$ is  infinite in Problem \eqref{firstproblem}, the computation of necessary and sufficient optimality conditions presents difficulties and requires a more delicate analysis than a simpler optimization problem. As far as we know, only  works related to infinite programming deal directly with infinite-many constraints (see, e.g., \cite{MR1234637,MR2295358} and the references therein). For that reason, it is necessary to solve an approximation of  Problem \eqref{firstproblem}. In this regard, the so-called \emph{scenario approach} emerges as a possible solution. The  \emph{scenario approach} corresponds to a min-max approximation of the original robust optimization problem using a sequence of samples.   It has used to provide an approximate solution to convex and nonconvex optimization problems (see, e.g., \cite{Campi_Garatti_Ramponi_2015,campi2009scenario,Care_Garatti_Campi_2015}). Furthermore,  the consistency of this method has been recently  provided in \cite{Campi2018} for convex optimization problems.

The intention of this work is to provide the consistency of the following method used to solve RO problems: Consider an  \emph{ergodic measure preserving transformation}  $T: \Xi \to \Xi$, then one can systematically  solve the  sequence of optimization problems
\begin{equation}\label{OP02ergodic2} \begin{aligned}
&\;\; \min g_n(x)\\
s.t. \;\; & x \in M(T^{k}(\xi)); \; k=1,\ldots,n,
\end{aligned}
\end{equation}
where $g_n$ is a sequence of functions, which converge continuously to the objective function $g$, and  $T^k$ represents the $k$-times composition  of $T$. Here, the desired conclusion is that the optimal value and the minimizers of  \eqref{OP02ergodic2}  converge, in some sense, to the optimal value and the minimizers of  \eqref{firstproblem} for almost all possible  choices of $\xi \in \Xi$. This conclusion is established in Corollary \ref{Corollary_ergodic}, which follows directly from our main result Theorem \ref{maintheorem01ergodic}. 

The key point in our results is to make a connection among three topics: (i)  the ideas of \emph{scenario approach}, (ii) an  \emph{ergodic theorem for random lower semicontinuous functions} established in  \cite[Theorem 1.1]{Korf_Wets_2001}, and (iii) the \emph{theory of epigraphical convergence of functions}.  After that, and due to the enormous developments in the theory of epi-convergence (see, e.g., \cite{Attouch_1984_book,Rockafellar_wets_book1998}), we can quickly establish some link between  the minimizers and the optimal value of the robust optimization problem \eqref{firstproblem} and its corresponding approximation \eqref{OP02ergodic2}.

As a consequence of this method,  we obtain the   consistency of the \emph{scenario approach}  for nonconvex optimization problems. More precisely, in this method one considers a drawing of independent and $\mathbb{P}$-distributed random function $\xi_1,\xi_2,\ldots$, and systematically  solves the sequence of optimizations problems.
\begin{equation}\label{OP020} \begin{aligned}
&\;\; \min g_n(x)\\
s.t. \;\; & x \in M(\xi_k); \; k=1,\ldots,n.
\end{aligned}
\end{equation} 
Again, the  conclusion relies on showing that the optimal value and the minimizers of  \eqref{OP020} converges to the solution of \eqref{firstproblem} for almost all possible 
sequences $(\xi_1,\xi_2,\ldots)$.

It is worth mentioning that our method allows us to  solve nonconvex optimization problems and to consider a perturbation of the objective function $g$ in \eqref{OP02ergodic2} and \eqref{OP020}, which is not guaranteed by the results of \cite{Campi2018}. Here, it has not escaped our notice   that the perturbation of  $g$ could be useful to ensure smoothness of the objective function in \eqref{OP02ergodic2} and \eqref{OP020}. On the other hand, the functions $g_n$ could be used to guarantee the existence and uniqueness of the numerical solutions of  \eqref{OP02ergodic2} and \eqref{OP020}.

The rest of the paper is organized as follows:  In Section \ref{notationandpreliminary}, we summarize the main definitions and notions using in the presented manuscript.  Next, in   Section \ref{consistency_Main}, we provide our main result, which is the consistency of the method presented in \eqref{OP02ergodic2}. In Section  \ref{SIP_SECTION}, first, we show that our result can be used to provide direct proof of the consistency of the \emph{scenario approach} for (even) nonconvex optimization problems, second, we show that our ergodic approach  can  be applied to problems related to infinite programming. In Section \ref{SectioNum}, we show some simple numerical examples of our results. Finally,  the paper ends with some conclusions and perspectives for future investigations.

\section{Notation and Preliminary}\label{notationandpreliminary}
In the following, we consider that $(X,d)$ is a Polish space, that is to say, a complete  separable metric space and $(\Xi,\mathcal{A},\mathbb{P})$  is a complete probability space. The \emph{Borel} $\sigma$-algebra on $X$ is denoted by $\mathcal{B}(X)$, which we recall  is the smallest $\sigma$-algebra containing all open sets of $X$. 

For a function $f:X \to \Rex$, a set $C\subseteq X$  and $\alpha \in \mathbb{R}$,  we define the $\alpha$-sublevel set of $f$ on $C$ as $$\textnormal{lev}_{\leq \alpha} (f,C):=\{  x\in  C: f(x) \leq \alpha		\},$$
when $C=X$, we omit the symbol $C$. We say that $f$ is lower semicontinuous  (lsc) if for all $\alpha \in \mathbb{R}$ the $\alpha$-sublevel set of $f$ on $X$ is closed. 

Following \cite{Attouch_1984_book}, let us consider a set $C \subseteq X$ and $\epsilon\geq 0$. We define the $\epsilon$-infimal value of $f$ on $C$ by
\begin{align*}
v_\epsilon(f,C):=\left\{ \begin{array}{cc}
\inf_C f + \epsilon, &\text{ if} \inf_C f >-\infty,\\
- \frac{1}{\epsilon}, &\text{ if } \inf_C f =-\infty.
\end{array}
\right.
\end{align*}
with the convention $\frac{1}{0}=+\infty$. We omit the symbol   $C$, or $\epsilon$ when $C=X$, or when $\epsilon=0$, respectively. Furthermore, we define the $\epsilon$-$\argmin$ of $f$ on $C$ by 
\begin{align*}
\epsilon\text{-}\argmin_C f := \left\{ x\in C:  f(x) \leq v_\epsilon(f,C)      \right\},
\end{align*}
again we omit the symbol   $C$, or $\epsilon$ when $C=X$ or when $\epsilon=0$, respectively.

For a set $A \subseteq X$, we define the indicator function of $A$, given by,
\begin{align*}
\delta_A(x):=\left\{ \begin{array}{cc}
0, &\text{ if } x\in A,\\
+\infty, &\text{ if } x\notin A.
\end{array}
\right.
\end{align*}

A function $f: \Xi \times X \to \Rex$ is called a \emph{random lower semicontinuous function}  (also called a \emph{normal integrand function}) if 
\begin{enumerate}[label=(\roman*)]
	\item the function $(\xi, x) \to f(\xi, x)$ is $\mathcal{A}\otimes \mathcal{B}(X)$-measurable, and 
	\item for every $\xi\in \Xi$ the function $f_\xi :=f(\xi, \cdot)$ is  lsc.
\end{enumerate}

Let us consider a set-valued map (also called a multifunction) $M: \Xi \tto X$. We  say that $M$ is measurable  if for every open set $U \subseteq  X$ the set 
$$M^{-1}(U):=\{  x \in X : M(x) \cap U\neq \emptyset   \} \in\mathcal{A}.$$
For more details about the theory of normal integrand  and measurable multifunctions we refer to \cite{Rockafellar_wets_book1998,MR0467310,MR2458436,MR1485775}.

Consider a sequence of sets $S_n \subseteq X$. We set $\liminf_{n\to \infty}  S_n$  and $\limsup_{n\to \infty} S_n$ as the inner-limit and the outer-limit, in the sense of  \emph{Painlev\'e-Kuratowski}, of the sequence $S_n$, respectively,  that is to say,
\begin{align*}
\liminf\limits_{n\to\infty} S_n &:= \left\{  x \in X:\limsup\limits_{n\to \infty} d(x,S_n) =0\right\} \\
\limsup\limits_{n\to\infty} S_n&:=\left\{  x \in X:\liminf\limits_{n\to \infty} d(x,S_n) =0\right\},
\end{align*}
where $d(x,S_n):=\inf \left\{ 	d(x,y) : y \in S_n			\right\}$.

Now, let us recall  some  notations about the convergence of functions. 
\begin{definition}
	Let  $f_n:X\to \Rex$ be a sequence of functions. The functions $f_n$ are said to epi-converge to $f$, denoted by $f_n\eto f$,  if for every $x\in X$
	\begin{enumerate}[label=\alph*)]
		\item $\liminf\limits_{n\to \infty} f_n(x_n) \geq f(x)$ for all $x_n \to x$.
		\item $\limsup\limits_{n\to \infty} f_n(x_n) \leq f(x)$ for some $x_n \to x$. 
	\end{enumerate}
\end{definition}
We refer to \cite{Attouch_1984_book,Rockafellar_wets_book1998} for more details about the theory of epi-graphical convergence. %, and several properties about the relation between  the  infimal value and the $\argmin$ of the approximate sequence of functions and  with respect to the minimizers of the epigraphical limit.

Also, we will need the following notation, which is equivalent to uniform convergence over compact sets for continuous functions (see, e.g., \cite{Rockafellar_wets_book1998}).

\begin{definition}
	We say that a sequence of functions $f_n:X\to \mathbb{R}\cup\{ +\infty\}$ converges continuously to $f$, if for every $x\in X$ and every $x_n \to x$
	\begin{align*}
	\lim\limits_{n\to \infty} f_n(x_n) =f(x).
	\end{align*}
\end{definition}

%A sequence of sets $(C_n)_{n\in \mathbb{N}}$ is eventually compact provided that there exists $n_0 \in \mathbb{N}$ such that for 
%\begin{align*}
%\bigcup\limits_{n \geq n_0} C_n \text{ is relatively compact}.
%\end{align*}
The following definition is an extension of the notation  \emph{eventually level-bounded}  used in finite-dimension setting, which can be found in \cite[Chapter 7.E ]{Rockafellar_wets_book1998}. We extend this notation as follows: Consider a sequence of functions $(f_n)_{n\in \mathbb{N}}$ and a sequence of sets $(C_n)_{n\in \mathbb{N}}$, we say that a sequence of functions $f_n$ is \emph{eventually level-compact on} $C_n$, if for each $\alpha \in \mathbb{R}$  there exists $n_\alpha\in \mathbb{N}$ such that  $$\bigcup\limits_{n \geq n_\alpha} \textnormal{lev}_{\leq \alpha}(f_n,C_n)  \text{ is relatively compact}.$$ 
In particular, if $C_n =X$, we simply say that $f_n$ is \emph{eventually level-compact}.

Now, we present two results. The first lemma shows that the sum of an epi-convergent sequence and a continuously convergent sequence   epi-convergences to the sum of limits.  The second proposition corresponds to a slight generalization of \cite[Proposition 2.9]{Attouch_1984_book} (see also \cite[Proposition 7.30]{Rockafellar_wets_book1998}), where only sequences $\epsilon_n \to 0$ were considered. For the sake of brevity we shall omit the proves.

\begin{lemma}\label{epigraphconvergencesum}
	Consider sequences of functions $p_n,q_n: X \to \Rex$ such that $p_n$ converges continuously to $p$ and $q_n$ epi-converges to $q$. Then,
	$p_n+q_n \eto p+q$.
\end{lemma}

\begin{proposition}\label{Proposition2.9_Attouch_1984_book}
	Let $f_n \eto f$ and $\epsilon_n \geq 0$ be a  sequence such that  $\epsilon =\limsup\epsilon_n<+\infty$. Then,
	\begin{enumerate}[label=\alph*)]
		\item $\limsup v_{\epsilon_n}(f_n) \leq v_\epsilon(f)$.
		\item 	$\limsup \epsilon_n\text{-}\argmin f_n  \subseteq \epsilon\text{-}\argmin f.$
	\end{enumerate}
\end{proposition}

\section{Consistency of  the  Approach to Robust Optimization Problems}\label{consistency_Main}
In this section we consider the following  optimization problem

\begin{equation}\label{OP01}\tag{$\mathcal{R}$}
\begin{aligned}
&\;\; \min g(x)\\
s.t. \;\;&x \in \eint{M},
\end{aligned}
\end{equation}
where $\eint{M}:=\{  x\in X: x \in M(\xi) \text{ a.s.}\}$,  $g: X\to \Rex$ is an lsc function and $M: \Xi \tto X$ is a measurable multifunction with closed values. We study  an approach using  \emph{ergodic measure preserving transformation}. Particularly,  we show the consistency of this method.

Now, we consider the following approach using \emph{ergodic measure  preserving transformation}. First, let us formally introduce this notion.
Consider a (complete) probability space  $(\Xi,\mathcal{A},\mathbb{P} )$ and a measurable function $T: \Xi \to \Xi$. We say that $T$ \emph{preserves measure} if
\begin{align}\label{pres_meas}
\mathbb{P}( T^{-1}(A) )= \mathbb{P}(A), \text{  for all }A\in \mathcal{A}.
\end{align}
Furthermore, we say that $T$ is \emph{ergodic} provided that for all $A\in \mathcal{A}$
\begin{align}\label{ergodic_trans}
A=T^{-1}(A) \Rightarrow \mathbb{P}(A)=0, \text{ or }  \mathbb{P}(A^c) =0.
\end{align}
Consequently, we say that $T$ is an  \emph{ergodic measure  preserving transformation} provided that  $T$ satisfies \eqref{pres_meas} and \eqref{ergodic_trans} .

We consider a sequence of lsc functions $g_n$, which converge continuously to $g$,  let us consider an  \emph{ergodic measure preserving transformation}  $T: \Xi \to \Xi$.  With this setting, we define the following family of optimization problems: For a point $\xi \in \Xi$ we define
\begin{equation}\label{OP02ergodic}\tag{$\mathcal{E}_{n}(\xi)$} \begin{aligned}
& \quad \quad \quad\min g_n(x)\\
s.t. \;\; & x \in E_n(\xi):=\bigcap\limits_{k=1}^n M(T^{k}(\xi)),
\end{aligned}
\end{equation}
where $T^{k}$ denotes the $k$-times composition of $T$.  In order to show  more clearly the link of epigraphical convergence and the relation  between \eqref{OP01} and \eqref{OP02ergodic}, let us define the functions 
$f_n : \Xi \times X \to \Rex$ and $f: \Xi \times X \to \Rex$ by
\begin{align}\label{notation_ergodic}\begin{aligned}
f_n(\xi, x)&:= g_n(x) + \frac{1}{n} \sum\limits_{k=1}^n\delta_{M(T^{k}(\xi))}(x),\\
f(x)&:= g(x) + \delta_{\eint{M}}(x).
\end{aligned}
\end{align}
With this notation we can write the relationship  between \eqref{OP01} and \eqref{OP02ergodic} in a functional formulation.
\begin{theorem}\label{maintheorem01ergodic}
	Under the above setting we have that $f_n(\xi,\cdot) \eto f$, $\mathbb{P}$-a.s. Consequently for any measurable sequence $\epsilon_n :\Xi^\infty \to (0,+\infty)$ with $\epsilon(\xi) :=\limsup \epsilon_n(\xi)<+\infty$,  $\mathbb{P}$-a.s.  we have that:
	\begin{enumerate}[label=\alph*)]
		\item $\limsup\limits_{n\to \infty }v_{\epsilon_n(\xi) } (f_n(\xi,\cdot) ) \leq  v_{\epsilon(\xi)} (f)$, $\mathbb{P}$-a.s.
		\item  
		$\limsup\limits_{n\to \infty} \epsilon_n(\xi)\text{-}\argmin f_n(\xi,\cdot) \subseteq \epsilon(\xi)\text{-}\argmin f,\;
		\mathbb{P}\text{-a.s.}$
	\end{enumerate}
\end{theorem}

\begin{proof}
	Let us consider the sequence of functions \begin{align*}
	p_n(x):=g_n(x) \quad \text{and} \quad q_n(\xi, x):=\frac{1}{n} \sum_{k=1}^n\delta_{M( T^{k}(\xi))}(x)
	\end{align*}. 
	
	It is not difficult to see that  the function $(\xi, x) \to \delta_{M(\xi)}(x)$ is a random lsc function and the function $\xi \to \inf   \delta_{M(\xi)}(x)$ is integrable. Then, by \cite[Theorem 1.1]{Korf_Wets_2001}, we have that $$q_n(\xi,\cdot) \eto \mathbb{E}_{\xi}(\delta_{M(\xi)}(\cdot))=\delta_{\eint{M}}(\cdot), \;\mathbb{P}\text{-a.s.}$$
	Now, define $\hat{\Xi}:=\{ \xi \in \Xi : q_n(\xi,\cdot) \eto \delta_{\eint{M}}   \}$, it follows that $\mathbb{P}(\hat{\Xi})=1$. Thus for all $\xi \in \hat{\Xi}$ we apply Lemma \ref{epigraphconvergencesum}, which implies that  for all $\xi  \in \hat{\Xi}$,   we have $p_n + q_n(\xi, \cdot) \overset{e}{\to } f$, that is to say, 
	$f_n(\xi, \cdot ) \overset{e}{\to} f $ for all  $\xi \in \hat{\Xi}$.
	
	Now, by Proposition \ref{Proposition2.9_Attouch_1984_book} we have that for all $\xi  \in \hat{\Xi}$,
	\begin{align*}
	\limsup\limits_{n\to \infty }v_{\epsilon_n(\xi) } (f_n(\xi,\cdot) ) \leq  v_{\epsilon(\xi)} (f), \text{ and}\\
	\limsup\limits_{n\to \infty} \epsilon_n(\xi)\text{-}\argmin f_n(\xi,\cdot) \subseteq \epsilon(\xi)\text{-}\argmin f,
	\end{align*}
	which concludes the proof.
\end{proof}

 When there are additional assumptions about  the feasibility and compactness of the optimization problems  \eqref{OP01} and \eqref{OP02ergodic} we can establish  a tighter conclusion. We translate the hypothesis into notation of the problems \eqref{OP01} and \eqref{OP02ergodic}, respectively.
\begin{corollary}\label{Corollary_ergodic}
	Let us assume that  \eqref{OP01} is feasible, and  the sequence of function $g_n$  is eventually level compact on $E_n(\xi)$   $\mathbb{P}$-a.s. $\xi \in \Xi$.	Then,
	\begin{enumerate}[label=\alph*)]
		\item	$\lim\limits_{n\to \infty }v( g_n, E_n(\xi) ) =  v(g,M_{as})$, $\mathbb{P}$-a.s.
		\item For any measurable sequence $\epsilon_n :\Xi \to (0,+\infty)$ with $\epsilon_n(\xi) \to 0$, $\mathbb{P}\text{-a.s.}$ we have that 
		\begin{align*}
		\emptyset \neq \limsup\limits_{n\to \infty} \epsilon_n(\xi)\text{-}\argmin\limits_{ E_n(\xi)} g_n \subseteq \argmin\limits_{M_{as}} g  ,\;
		\mathbb{P}\text{-a.s.}
		\end{align*}
		\item $
		\bigcap\limits_{\epsilon >0 } \liminf\limits_{n\to \infty} \epsilon\text{-}\argmin\limits_{E_n(\xi)} g_n  = \argmin\limits_{M_{as}}  g  = \bigcap\limits_{\epsilon >0 } \limsup\limits_{n\to \infty} \epsilon\text{-}\argmin\limits_{E_n(\xi)} g_n ,\; \mathbb{P}\text{-a.s.}
	$
	\end{enumerate}
\end{corollary}
\begin{proof}
	Consider the notation given  in \eqref{notation_ergodic}. Let us define $\alpha:=\max\{ \inf  f +1,1\}$, we have that $\alpha< +\infty$ due to the feasibility of    \eqref{OP01}.  Consider a set $\hat{\Xi}$ of full measure such that  for all $\xi \in \hat{\Xi}$
	\begin{enumerate}[label=(\roman*)]
		\item  $(f_n(\xi,\cdot))_{n\in \mathbb{N} }$  is eventually level compact,
		\item $\epsilon_n(\xi) \to 0$,
%		\item $x_n(\xi) \in \epsilon( \xi)\text{-}\argmin f_n(\xi, \cdot)$,
		\item $\limsup\limits_{n\to \infty} \epsilon_n(\xi)\text{-}\argmin f_n(\xi, \cdot) \subseteq \argmin f$.
	\end{enumerate}
	Fix  $\xi \in \hat{\Xi}$ and a sequence $x_k(\xi) \in \epsilon_{n_k}( \xi)\text{-}\argmin f_{n_k}(\xi, \cdot)$, so there exists some $n_{\xi} \in \mathbb{N}$ such that for all $n \geq n_\xi$
	$$\epsilon_n(\xi) \leq 1,\text{ and }\bigcup\limits_{n \geq n_\xi} \textnormal{lev}_{\leq \alpha} f_n \text{ is relativelly compact.}$$ This implies that the sequence $(x_k(\xi))_{k\geq n_\xi}^\infty$ belongs to a compact set, so it has an accumulation point.	Consequently, we have that   $\limsup\limits_{n\to \infty} \epsilon_n\text{-}\argmin f_n(\xi, \cdot)\neq \emptyset,$ which proves  \ref{corollaryb}.
	
		Now, by  \cite[Theorem 2.11]{Attouch_1984_book} we conclude that
	$\lim\limits_{n\to \infty }v( f_n(\xi, \cdot))  =  v( f)$ for all $\xi \in \hat{\Xi}$, which concludes the proof of \ref{corollarya}. Finally, using \cite[Theorem 2.12]{Attouch_1984_book} we get that  $c)$ holds.
	
\end{proof}

\section{Applications}
    In this section, we present applications of the result found in Section \ref{consistency_Main}. The first application, given in Subsection \ref{subsection_scenario}, corresponds to prove the consistency of the  nonconvex \emph{scenario approach}  (see, e.g., \cite{Campi_Garatti_Ramponi_2015,campi2009scenario,Care_Garatti_Campi_2015}). The second application, given in Subsection  \ref{SIP_SECTION}, shows that our ergodic approach can be used to solve infinite programming problems.

\subsection{Consistency of Nonconvex Scenario Approach}\label{subsection_scenario}
%In this section we consider a polish space $X$ and $D$ is a partially order polish space. 
In this section we consider  $ (\Xi^\infty, \mathcal{A}^\infty,\mathbb{P}^\infty)$ as the  denumerable product of the probability  space $(\Xi,\mathcal{A},\mathbb{P})$.

As in the previous section, we consider a sequence of lsc functions $g_n$, which converge continuously to $g$, let us define the following family of optimization problems: For each $\omega=(\xi_k)_{k=1}^\infty \in \Xi^\infty$ we set
\begin{equation}\label{OP02}\tag{$\mathcal{S}_{n}(\omega)$} \begin{aligned}
& \quad \quad \quad \min g_n(x)\\
s.t. \;\; & x \in S_n(\omega):=\bigcap\limits_{k=1}^{n}  M(\xi_k),
\end{aligned}
\end{equation} 
Let us define
$f_n : \Xi^\infty \times X \to \Rex$ given by 
\begin{align}\label{f_n}
\begin{aligned}
f_n(\omega, x)&:= g_n(x) + \frac{1}{n} \sum\limits_{k=1}^n\delta_{M(\xi_k)}(x),\\
f(x)&:= g(x) + \delta_{\eint{M}}(x).
\end{aligned}
\end{align}
The following results corresponds to the \emph{scenario approach} version of Theorem \ref{maintheorem01ergodic}.
\begin{theorem}\label{maintheorem01}
	Under the above setting we have that $f_n(\omega ,\cdot) \eto f$, $\mathbb{P}^\infty$-a.s. Consequently for any measurable sequence $\epsilon_n :\Xi^\infty \to (0,+\infty)$ with $\epsilon(\omega) :=\limsup \epsilon_n(\omega)<+\infty$,  $\mathbb{P}^\infty$-a.s.  we have that:
	\begin{enumerate}[label=\alph*)]
		\item $\limsup\limits_{n\to \infty }v_{\epsilon_n(\omega) } (f_n(\omega,\cdot) ) \leq  v_{\epsilon(\omega)} (f)$, $\mathbb{P}^\infty$-a.s.
		\item  
		$\limsup\limits_{n\to \infty} \epsilon_n(\omega)\text{-}\argmin f_n(\omega,\cdot) \subseteq \epsilon(\omega)\text{-}\argmin f,\;
		\mathbb{P}^\infty\text{-a.s.}$
	\end{enumerate}
	
\end{theorem}
\begin{proof}
	Consider the \emph{shift} on $\Xi^\infty$, that is, $T: \Xi^\infty \to \Xi^\infty$  given by
	\begin{align}\label{shif_T}
	T( (\xi_i)_{i=1}^\infty) =(\xi_{i+1})_{i=1}^\infty,
	\end{align}
	by \cite[Proposition 2.2]{Coudene_2016_book} $T$ is an  \emph{ergodic measure preserving  transformation} (For more details we refer to \cite{Walters_1982_book,Coudene_2016_book}).
	Furthermore, we extend the measurable multifunction $M$ to $\Xi^\infty$ just by defining $\tilde{M} : \Xi^\infty \tto X$ by  $\tilde{M}(\omega )=M(\xi_1)$, where $\omega =(\xi_i)_{i=1}^{\infty}$.  Using  notation  \eqref{f_n} we get 
	\begin{align}\label{6_samples}
	\begin{aligned}
	f_n(\omega, x)&:= g_n(x) + \frac{1}{n} \sum\limits_{k=1}^n\delta_{\tilde{M}(T^{k}(\omega))}(x)=g_n(x) + \frac{1}{n} \sum\limits_{k=1}^n\delta_{M(\xi_k)}(x),\\
	f(x)&:= g(x) + \delta_{\eint{\tilde{M}}}(x)=g(x) + \delta_{\eint{M}}(x)..
	\end{aligned}
	\end{align}
	Then, Theorem \ref{maintheorem01ergodic} gives us that for almost all $\omega=(\xi_i) _{i=1}^\infty \in \Xi^\infty$
	\begin{enumerate}[label=\roman*)]
		\item $ f_n (\omega, \cdot) \to f,$
		\item $\limsup\limits_{n\to \infty }v_{\epsilon_n( \omega ) } (f_n( \omega ,\cdot) ) \leq  v_{\epsilon(\omega) } (f)$, 
		\item $\limsup\limits_{n\to \infty} \epsilon_n(\omega)\text{-}\argmin f_n( \omega,\cdot) \subseteq \epsilon( \omega )\text{-}\argmin f$
	\end{enumerate}
\end{proof}
\begin{remark}
	It is worth mentioning that Theorem \ref{maintheorem01} can be proved using the same proof given in Theorem \ref{maintheorem01ergodic},  copied step by step, but  using \cite[Theorem 2.3]{Arstein_Wets_1995} instead of \cite[Theorem 1.1]{Korf_Wets_2001}. 
\end{remark}

Similar to the previous section, we can get more precise estimations under some compactness assumptions. The proof of this result follows considering the  representation of \eqref{f_n} given in \eqref{6_samples} using the shift transformation defined in \eqref{shif_T}.  Also, it can follow mimicking the proof of the Corollary step by step, and using Theorem \ref{maintheorem01} instead of Theorem \ref{maintheorem01ergodic}.
\begin{corollary}\label{corollaryscneario}
	Let us assume that  \eqref{OP01} is feasible, and the  sequence of functions $g_n$  is eventually level compact on $S_n(\omega)$  $\mathbb{P}^\infty$-a.s. $\omega \in \Xi^\infty$. Then,
	\begin{enumerate}[label=\alph*)]
		\item \label{corollarya}	$\lim\limits_{n\to \infty }v ( g_n ,S_n(\omega)) =  v(g,M_{as})$, $\mathbb{P}^\infty$-a.s. $\omega \in \Xi^\infty$.
		\item\label{corollaryb}For any measurable sequence $\epsilon_n :\Xi^\infty \to (0,+\infty)$ with $\epsilon_n \to 0$, $\mathbb{P}^\infty\text{-a.s.}$ we have that 
		\begin{align*}
		\emptyset \neq \limsup\limits_{n\to \infty} \epsilon_n(\omega)\text{-}\argmin\limits_{ S_n(\omega) } g_n  \subseteq \argmin\limits_{M_{as}} g ,\;
		\mathbb{P}^\infty\text{-a.s. } \omega \in \Xi^\infty.
		\end{align*}
		\item $
		\bigcap\limits_{\epsilon >0 } \liminf\limits_{n\to \infty} \epsilon\text{-}\argmin\limits_{ S_n(\omega) } g_n= \argmin\limits_{M_{as}} g    = \bigcap\limits_{\epsilon >0 } \limsup\limits_{n\to \infty} \epsilon\text{-}\argmin\limits_{ S_n(\omega) } g_n,\; \mathbb{P}^\infty\text{-a.s. } \omega \in \Xi^\infty.
	$
	\end{enumerate}
\end{corollary}

\begin{remark}
	It has not escaped our notice that in \cite{Ramponi2018} the authors did not show the consistency of the \emph{scenario approach} with a perturbation over the objective function $g$ as in \eqref{OP02}. Furthermore, only linear objective function and convex constraint sets were considered in \cite{Ramponi2018}.  
\end{remark}
{ %\color{red}
 In the next result, we provide a concrete application of the above corollary using a  Moreau envelope of the objective function. Following \cite{Rockafellar_wets_book1998}, we recall that  given a function $g$ and $\lambda>0$, the Moreau envelope function $\MoreauYosida{g}{\lambda}$ is defined by
\begin{align*}
\MoreauYosida{g}{\lambda}(x):=\inf\left\{ 	f(u) + \frac{1}{2\lambda }\| x -u\|^2	 : u \in \mathbb{R}^d		\right\}.
\end{align*}
A function $g$ is said to be prox-bounded if there exists $\lambda >0$ such that $\MoreauYosida{g}{\lambda}(x)>-\infty$ for some $x\in \mathbb{R}^n$. In that case the supremum of all such $\lambda $ is the threshold $\lambda_g$ of prox-boundeness for $g$.

\begin{corollary}
	Let $g: \mathbb{R}^d \to \Rex $ be a proper, lsc and prox-bounded function with threshold $\lambda_g >0$, let $h :\Xi \times \mathbb{R}^d \to \Rex$ be a normal integrand function and $C \subseteq \mathbb{R}^d$ be a bounded closed set, and defined the optimization problem
	\begin{equation}\label{exact}
	\begin{aligned}
&  \min g(x)\\
s.t. \;\; & \varphi(\omega,x)\leq 0, \text{ a.s.}\\
\;\; & x\in C.
\end{aligned}
	\end{equation}
Consider  $\lambda_n  \searrow 0$ with $\lambda_n   \in (0, \lambda_g)$. For each  $\omega=(\xi_k)_{k=1}^\infty \in \Xi^\infty$ we defined the sequence of optimization problems 
	\begin{equation}\label{aproxn}
\begin{aligned}
&  \min \MoreauYosida{g}{\lambda_n}(x)\\
s.t. \;\; & \varphi(\xi_i,x)\leq 0, \text{ for }i=1,\ldots, n\\
\;\; & x\in C.
\end{aligned}
\end{equation}
 Let $x_n(\omega) $ be a (measurable) selection of the minimizers of Problem \eqref{aproxn}. Then, if Problem \eqref{exact} is feasible, we have
 \begin{enumerate}[label=\alph*)]
 	\item $\MoreauYosida{g}{\lambda_n}(x_n(\omega))$ converges to the optimal value of Problem \eqref{exact},  $\mathbb{P}^\infty$-a.s. $\omega \in \Xi^\infty$.
 	\item Any cluster point of $(x_n(\omega))_{n\in \mathbb{N}}$ is a minimizer of Problem \eqref{exact}, $\mathbb{P}^\infty$-a.s. $\omega \in \Xi^\infty$.
 \end{enumerate}  
	\end{corollary}
\begin{proof}
Let us define the  multifunction $M: \Xi \tto C$  given by $M(\xi):=\{ x\in C: \varphi(\xi,x) \leq0\},$ which is measurable due to  \cite[Proposition 14.33]{Rockafellar_wets_book1998}. Moreover, 
by theorem \cite[Theorem 1.25]{Rockafellar_wets_book1998}, $\MoreauYosida{g}{\lambda_n}$ converges continuously $g$. Finally, since $C \subseteq \mathbb{R}^d$ is bounded and closed,   we have that the sequence of function  $\MoreauYosida{g}{\lambda_n}$  is eventually level compact on the sets $\cap_{k=1}^n M(\xi_k)$. Therefore, applying Corollary \ref{corollaryscneario} we get the result.
	\end{proof}
}
\subsection{Application to Infinite Programming Problems}\label{SIP_SECTION}
In this part of the work, we use the result of Section  \ref{consistency_Main} to show that  a sequence of sub-problems can be used to give an approach for infinite programming problems.

Consider the following problem of \emph{infinite programming} (\emph{semi-infinite programming}, if $X$ is  a finite dimensional vector space)
\begin{equation}\label{OPSIP}\tag{$\mathcal{I}$}
\begin{aligned}
&\quad\quad  \min g(x)\\
s.t. \;\;&x \in M_a:=\bigcap\limits_{s \in S} M(s),
\end{aligned}
\end{equation}
where $S$ is a topological space, and $M:S \tto X$ is an  \emph{outer-semicontinuous} set-valued map, that is to say, for every net $s_\nu \to s$ and every net $x_\nu \in M(s_\nu)$ with $x_\nu \to x$ we have $x\in M(s)$.  We denote by $\mathcal{A}$  any $\sigma$-algebra, which contains all open subsets on $S$, and consider $\mu: \mathcal{A} \to \mathbb{R}$ a strictly positive finite measure, that is to say,   $\mu(S)<+\infty$ and 
\begin{align*}
\mu(U) >0, \text{ for every open set } U \subseteq S,
\end{align*}
let us consider an \emph{ergodic measure preserving transformation}  $T: S \to S$. With this framework, we define the sequence of optimization problems
\begin{equation}\label{OPSIP_n}\tag{$\mathcal{I}_{n}(s)$} \begin{aligned}
&\;\;\;\;\;\;\; \min g_n(x)\\
s.t. \;\; & x \in I_n(s):=\bigcap\limits_{k=1}^n M(T^{k}(s)),
\end{aligned}
\end{equation}
where $g_n$ converges continuously to $g$. As a simple application of  Theorem \ref{maintheorem01ergodic} we get the following result, which give us a relation between Problems \eqref{OPSIP_n} and   \eqref{OPSIP}.
\begin{corollary}\label{TheoreSEMI}
	Let us assume that $(S,\mathcal{A},\mu)$ is complete,   \eqref{OPSIP} is feasible, and   the  sequence of functions $g_n$  is eventually level compact on $I_n(s)$  $\mu$-a.e.	Then,
	\begin{enumerate}[label=\alph*)]
		\item 	$\lim\limits_{n\to \infty }v(g_n,I_n(s))  =  v(g,M_a)$, $\mu$-a.e.
		\item For any measurable sequence $\epsilon_n :S \to (0,+\infty)$ with $\epsilon_n(s) \to 0$, $\mu$-a.e. we have that 
		\begin{align*}
		\emptyset \neq \limsup\limits_{n\to \infty} \epsilon_n(s)\text{-}\argmin\limits_{I_n(s)} g_n  \subseteq \argmin\limits_{M_a} g ,\;
		\mu\text{-a.e.}
		\end{align*}
		\item $
		\bigcap\limits_{\epsilon >0 } \liminf\limits_{n\to \infty} \epsilon\text{-}\argmin\limits_{I_n(s)} g_n = \argmin\limits_{M_a} g   = \bigcap\limits_{\epsilon >0 } \limsup\limits_{n\to \infty} \epsilon\text{-}\argmin\limits_{I_n(s)} g_n ,\;  \mu\text{-a.e.}
	$
	\end{enumerate}
\end{corollary}

\begin{proof}
	First, by  the \emph{outer-semicontinuous} of $M$ we have that the optimization problem \eqref{OPSIP} is equivalent to 
	\begin{equation}\label{OPSIP_mu}
	\begin{aligned}
	&\;\; \min g(x)\\
	s.t. \;\;&x \in M(s), \;  \mu\text{-a.e.}
	\end{aligned}
	\end{equation}
	Indeed, let $x \in  M(s)$ for almost all $s\in S$. Then, the set $D_x :=\{ s\in S: x\in M(s)		\}$ is dense due to the fact that $\mu$ is a strictly positive  measure. Consequently, for every $s \in S\backslash D$ there exists $s_\nu \to s$, so by the  \emph{outer-semicontinuous} of $M$ we get that $x \in M(s)$, and consequently $x\in M(s)$ for all $s\in S$.
	
	Next, consider the probability measure $\mathbb{P}(\cdot)= \mu(S)^{-1} \mu (\cdot)$. Then, applying  Corollary \ref{Corollary_ergodic} to \eqref{OPSIP_mu} we get that $a), b)$ and $c)$ hold with  \eqref{OPSIP_mu}, and by the equivalency with  \eqref{OPSIP} we conclude the proof.
\end{proof}
{%\color{red}
To end this section, we present a result that considers a sequence of functions mollified by convolution. Consider a  sequence mollifiers, that is, a sequence of  measurable functions $e_n : \mathbb{ R }^d \to [0,+\infty)$  with $\int_{\mathbb{R}^d} e_n  (z) dz=1$, such that the sets $\{ z : e_n(z) >0\}$ decrease to $\{ 0\}$. Given a continuous function $g: \mathbb{ R }^n \to R$, we define its mollification, $g_n : \mathbb{ R }^d \to \mathbb{ R }$, by
\begin{align}\label{mollifications}
g_n( x) = \int_{\mathbb{ R }^d}  e_n(x-z)g(z) dz, \text{  for all }x\in \mathbb{ R }^d.
\end{align}
\begin{corollary}
 Let $g: \mathbb{R}^d \to \mathbb{R} $ and  $\varphi :S\times \mathbb{R}^d \to \mathbb{ R }$ be continuous functions, and $C \subseteq \mathbb{R}^d$ be a bounded closed set. Consider the following semi-infinite programming problem
	\begin{equation}\label{exactsemi}
	\begin{aligned}
	&  \min g(x)\\
	s.t. \;\; & \varphi(S,x)\leq 0, \text{ for all }s\in S,\\
	\;\; & x\in C.
	\end{aligned}
	\end{equation}
Given  a measure preserving transformation $T: S \to S$ and    $s \in S$, we defined the sequence of optimization problems 
	\begin{equation}\label{aproxn2}
	\begin{aligned}
	&  \min g_n(x)\\
	s.t. \;\; & \varphi(T^i(s),x)\leq 0, \text{ for }i=1,\ldots, n\\
	\;\; & x\in C,
	\end{aligned}
	\end{equation}
	where   $g_n$ are given by \eqref{mollifications}. 
	Let $x_n(s) $ be a (measurable) selection of the minimizers of Problem \eqref{aproxn2}. Then, if Problem \eqref{exactsemi} is feasible, we have
	\begin{enumerate}[label=\alph*)]
		\item $g_n(x_n(z))$ converges to the optimal value of Problem \eqref{exact},  $\mu$-a.e. $s\in S$.
		\item Any cluster point of $(x_n(z))_{n\in \mathbb{N}}$ is a minimizer of Problem \eqref{exactsemi}, $\mu$-a.e. $s \in S$.
	\end{enumerate}  
\end{corollary}
\begin{proof}
Define the multifunction $M : S\tto  \mathbb{ R }^d$ given by $ M(s):=\{ x\in C : \varphi(s,x) \leq 0\}$, which is measurable and outer-semicontinuous due to the continuity of $\varphi$. Now, by \cite[Exercise 7.19]{Rockafellar_wets_book1998}, we have that $g_n$ converges continuously to $g$. Finally, since $C$ is bounded, we have that $g_n$ is eventually level compact on $\cap_{i=1}^n M(T^{i}(s))$. Therefore, by Corollary \ref{TheoreSEMI}, we get the result.
	\end{proof}
}
\section{Numerical Examples}\label{SectioNum}
Now, let us illustrate the above result with two different   examples.  The first one consider a \emph{best polynomial approximation}, which in particular can be expressed as a convex optimization problem. The second one consider a non-convex optimization problem. 

\subsection{Best  Functional Approximation}
In this subsection we focus on the following optimization problem. Consider a (measurable) absolutely bounded function $f:[0,1] \to \mathbb{R}$ and a (finite) family of linearly independent absolutely bounded  functions $e_j : [0,1] \to \mathbb{R}$ with $j=0,\ldots,q$. We want to find the best approximation of  $f$ in the linear space spanned by $\{ e_j\}_{j=0 }^q$. In order to solve this problems we follow   \cite{MR1628195}. Let us consider the following optimization problem:

\begin{equation*}
\begin{aligned}
&\;\;\;\;\;\;\; \min  \quad \beta \\
 \text{ s.t. } &\;\;\;\;\;\;\;  \| f - \sum\limits_{ j=0 }^q x_i \cdot e_j \|_{\infty}\  \leq \beta.
\end{aligned}
\end{equation*}
It can be equivalently expressed as

\begin{equation}\label{bestpolynomialaprx}\tag{$\mathcal{B}$}
\begin{aligned}
&\;\;\;\;\;\;\;  \min  \quad g(x) \\
\text{ s.t. }  &\;\;\;\;\;\;\;  x \in M(t) \text{ a.e. } t\in [0,1].
\end{aligned}
\end{equation}
where the (measurable) set-valued map $M: [0,1] \to \mathbb{R}^{q+2}$ is given by
 $$M(t):=\left\{  (x,\beta) 	\in \mathbb{R}^{q+1} \times \mathbb{R} : 
-\beta \leq  f(t) - \sum\limits_{ j=0}^q x_i \cdot e_j (t)   \leq \beta	\right\}$$ and  $g: \mathbb{R}^{q+2} \to \mathbb{R}$ is given by $g(x,\beta)=\beta$.

To illustrate our results let us solve numerically \eqref{bestpolynomialaprx}  for the particular function    $$f(t)= 10^2 t(2t-1)(t -1)(4t -1)(4t -3)$$ and $e_j$ the canonical base of polynomials, that is, $e_j(t):=t^i$. We use an ergodic measure preserving transformation $T(t)=t + \alpha \mod 1$, with some $\alpha \notin \mathbb{Q}$, and systematically, we solve the sequence of optimization problems for a fixed point $\bar{t} \in [0,1]$
\begin{equation}\label{bestpolynomialaprxn}\tag{$\mathcal{B}_n(\bar{t})$}
\begin{aligned}
&\;\;\;\;\;\;\;  \min  \quad g(x) \\
\text{ s.t. }  &\;\;\;\;\;\;\;  x \in M(T^k(\bar{t})),\; k=1,...,n.
\end{aligned}
\end{equation}
In Figure \ref{figure02} we show the results of the polynomial approximation found solving Problem \eqref{bestpolynomialaprxn} for different values of $n$ and for point $\bar{t}=0$  and $\alpha=\sqrt{7}$.
	\begin{center}
	\begin{figure}[h!!!]
		\centering	\includegraphics[scale=0.37]{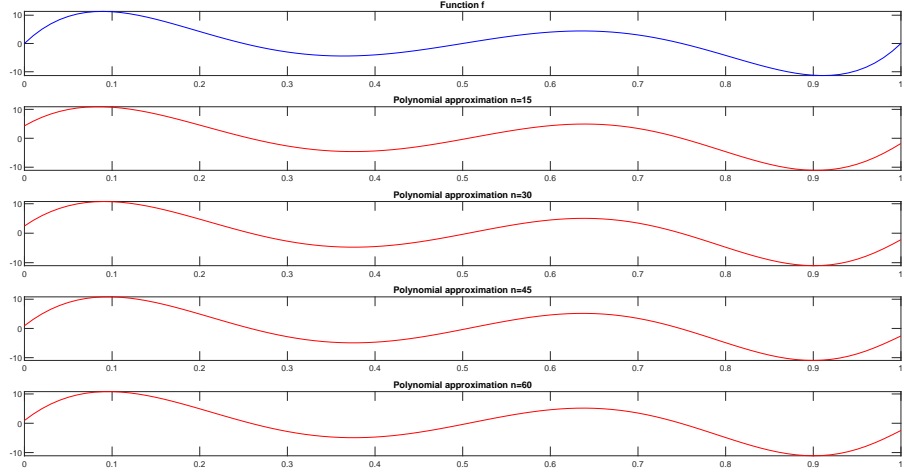}
		\caption{Polynomial approximation with $q=10$, $\bar{t}=\sqrt{5}$ and  $\alpha=\sqrt{7}$} \label{figure02}
	\end{figure}
\end{center}
\subsection{Rotation on the  $2$-dimensional Unit Sphere}\label{example}

Let us consider the following optimization problem
\begin{equation*}%\label{P_example}\tag{Ex}
\begin{aligned}
&\min g(x,y)&\\
\text{s.t. } &\alpha x + \beta y  \leq 6\;& \text{ for all } (\alpha,\beta) \in \mathbb{S}^{1},
\end{aligned}
\end{equation*}
where $g$ is a non-convex polynomial function with several local-minima  (see Figure \ref{figure01}), more precisely  we choose
\begin{align*}
g(x,y)= (x+5)(x+2)  (x-1) (x-9) +y(y+11)( y-4) ( y-5)  + xy. 
%x^4-3 x^3-51 x^2-37 x+y^4+2 y^3-79 y^2+220 y+90+xy
\end{align*} and $\mathbb{S}^{1}$ is the $2$-dimensional unit sphere, that is to say, $\mathbb{S}^{1}:=\{ (\alpha,\beta) \in \mathbb{R}^2	: \alpha^2 + \beta^2=1	\}$.
It is not difficult to see that the above problem is noting more than
\begin{equation}\label{P_example_real}
\begin{aligned}
&\min g(x,y)\\
\text{s.t. }  &x^2 +  y^2  \leq 36.
\end{aligned}
\end{equation}
To solve this problem, we use an irrational rotation  $T: \mathbb{S}^1 \to \mathbb{S}^1$, that is, $T(\xi)=\xi \cdot e^{2\pi\theta i}$ with $\theta \in [0,1] \backslash \mathbb{Q}$. Here the multiplication is in the sense of complex numbers. Therefore, we have to numerically solve  the following optimization problems

\begin{equation}\label{P_example_n} \tag{$\mathcal{U}_n(\xi)$}
\begin{aligned}
&\min g(x,y)&\\
\text{s.t. } &x \in U_n(\xi)
\end{aligned}
\end{equation}
where $ U_n(\xi): =\{  (x,y) \in \mathbb{R}^{
	2} :  \langle  T^k(\xi)  , (x,y)\rangle  \leq 6, \; \forall k=1,\ldots,n \}$.

First, we have that the global minimum of $g$ is attained at $(x_u,y_u)=(6.3442 ,  -7.6398)$ and the minimum is $g(x_u,y_u)=-5490.9$.  On the other hand the optimal value  of  \eqref{P_example_real}  is attained at $(x_c,y_c)=( 3.2004 ,  -5.0752)$ with value $g(x_c,y_c)=-3519.1$. In Table \ref{Table_num} we can compare different numerical solutions to Problem \eqref{P_example_n}.

	\begin{center}
	\begin{figure}[h!!!]
		\centering	\includegraphics[scale=0.5]{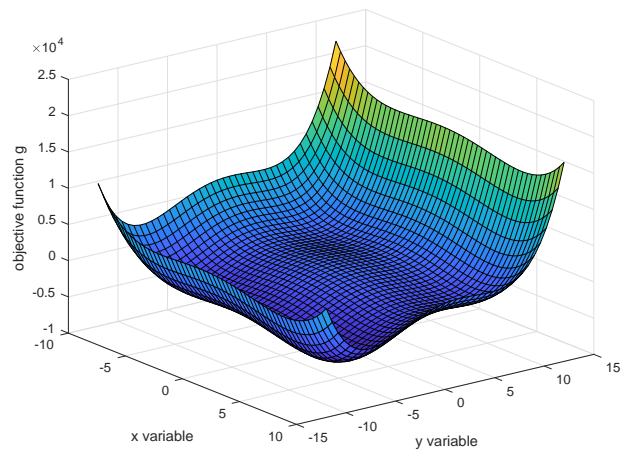}
		\caption{Function $g$ of Section \ref{example}} \label{figure01}
	\end{figure}
\end{center}

\begin{table}[h!!!]
	\begin{center}	
		\centering	\begin{tabular}{|||c|c|c|c|c|c|c|c|c|||}
			
			\hline	$\xi$ &$ n$ &$v(g,U_n(\xi))$ & $x$ & $y$ &$ n$ &$v(g,U_n(\xi))$  & $x$ & $y$ \\\hline \hline
			$e^{i{\pi}/{10}}$   &  5 &-3862.4 & 1.7623 & -6.7866  &15 &  -3634.7&  3.3542   &-5.1543        \\ \hline
			$e^{i{\pi}/{6}} $   & 5 &-4257.4 &2.6753&   -6.9991  &15 & -3519.3 &    3.0438  & -5.1746 \\ \hline
			$e^{i{\pi}/{3}} $   &  5 &-5406.3 &5.7143&     -7.2882 & 15 		&-3651.6	&		 4.3877  & -4.5573	\\ \hline
			$e^{i{4\pi}/{3}}$   &  5 &-4258.3 &  4.3254  & -5.5167&15 &	-3633.0			& 3.1903  & -5.2573	\\ \hline
			$e^{i{7\pi}/{4}}$   &  5 &-5251.2 &4.9946  & -7.5369 &15 & 	-3745.8	& 3.9081  & -4.9747
			\\ \hline

			$e^{i{\pi}/{10}}$      &  30 &-3519.4 & 3.3562 & -4.9778 &  40 &-3519.4 & 3.3562  & -4.9778 \\ \hline
			$e^{i{\pi}/{6}} $    &  30 &-3519.3 &3.0438  & -5.1746  &  40 &-3519.3 &3.1230   &-5.1243 \\ \hline
			$e^{i{\pi}/{3}} $   &  30 &-3542.3 &3.1288  & -5.1560&  40 &-3542.3 &3.0304&   -5.2028  \\ \hline
			$e^{i{4\pi}/{3}}$    &  30 &-3534.0 &  3.6424&   -4.8220& 40 &-3534.0 &3.5479  & -4.8771 \\ \hline
			$e^{i{7\pi}/{4}}$   &  30 &-3524.0 &3.6397   &-4.8089&  40 &-3524.0 &3.6397 &  -4.8089 \\ \hline
			
			$e^{i{\pi}/{10}}$      & 70 & -3519.4 &  3.3562   &-4.9778	&100	&		-3.5194 & 3.3562&   -4.9778 \\ \hline
			$e^{i{\pi}/{6}} $  & 70 &  -3519.3 &  3.1232  & -5.1242&100	&	-3.5193&  3.1232  & -5.1242
			\\ \hline
			$e^{i{\pi}/{3}} $  & 70 &-3532.1 &3.0304  & -5.2028&100	&-3.5219 &	2.9323  & -5.2494
			\\ \hline
			$e^{i{4\pi}/{3}}$  & 70 &-3529.4 &	3.4538 &  -4.9319&100	&-3.5268&	3.3599 &  -4.9866
			\\ \hline
			$e^{i{7\pi}/{4}}$  & 70 & -3523.6  &	3.5467   &-4.8655&100	&-3.5228 & 	3.4538  & -4.9220 \\ \hline
			
		\end{tabular}
		\captionsetup{justification=centering}
		\caption{Numerical solution of Problem \eqref{P_example_n}.}\label{Table_num}
	\end{center}
\end{table}

\newpage
\section{Conclusion and Perspectives}
In this paper, we have studied the consistency of a new method for solving robust optimization problems. In counterpart to classical methods in stochastic programming, it is based on ergodic measure preserving transformation instead of sample approximation.

In particular, our approach is more general, because we can recover the results based on samples using the  \emph{shift} on the denumerable product of the probability space.  Moreover, our results allow us to apply the technique to infinite programming problems under reasonable assumptions, and without any compactness assumption on the index set, as classical results in this field.

We believe that our analysis represents a first step in the understanding of a new approach based on ergodicity instead of a sequence of samples. A natural question relies on to understand the relation between the choice of the measure preserving transformation and the rate of convergence of the approximate sequence of minimizers.
 
{%\color{red}
It is important to mention that in our result a regularization of the objective function is considered in each step, that is the role of the sequence $g_n$ in Theorem \ref{maintheorem01ergodic} and its corollaries.  In most important applications, constraints are given by systems of possible nonsmooth inequalities. Commonly, this nonsmoothness is handled by using smooth regularizations, for instance, Moreau envelopes and regularization via mollifiers (see, e.g., \cite{Rockafellar_wets_book1998}). Therefore, it will be necessary the study of the consistency of a method which in each step use both a regularization of the constraints and our ergodic approach.

Therefore, we plan to investigate these possible  extensions of our presented method in a future paper.}

\textbf{Acknowledgements} First, the author would like to acknowledge the helpful discussions and great comments about this work given by  Professor  Marco A. L\'opez C\'erda, which improved notably the quality of the presented manuscript. Second, the author is grateful to the anonymous reviewers for their valuable suggestions and comments about the work, which increases the presentation of the current version of the manuscript.

\bibliographystyle{plain}
\bibliography{ErgodicApproach20200909}
\end{document}